\documentclass[10pt, reqno]{amsart}
\usepackage{amsmath, amsthm, hyperref}
\usepackage{times}
\usepackage{amsfonts}
\usepackage{amssymb}
\usepackage{xcolor} 
\usepackage{graphicx} 
\usepackage{cite} 

\usepackage{enumitem}
\setenumerate{itemsep=5pt} 
\setitemize{itemsep=5pt, leftmargin=10pt} 

\newcommand{\mdim}{\operatorname{dim_{\mathrm{mat}}}}
\newcommand{\Ap}{\operatorname{Ap}}
\newcommand{\ord}{\operatorname{ord}}
\newcommand{\M}{\mathsf{M}}
\newcommand{\MM}{\mathcal{M}}
\newcommand{\N}{\mathbb{N}}
\newcommand{\Q}{\mathbb{Q}}
\newcommand{\semigroup}[1]{\langle #1 \rangle}
\newcommand{\s}{\mathfrak{s}}
\newcommand{\vecspan}{\operatorname{span}}

\renewcommand{\S}{\mathcal{S}}
\newcommand{\V}{\mathcal{V}}
\newcommand{\Z}{\mathbb{Z}}
\newcommand{\0}{{\color{lightgray}0}}

\renewcommand{\vec}[1]{\mathbf{#1}}
\renewcommand{\pmod}[1]{\,\,(\operatorname{mod} #1)}

\newcommand{\tinymatrix}[4]{\big[ \begin{smallmatrix} #1 & #2 \\ #3 & #4 \end{smallmatrix} \big]}

\newtheorem{thm}{Theorem}
\newtheorem{prop}[thm]{Proposition}

\newtheorem{lemma}[thm]{Lemma}

\theoremstyle{definition}
\newtheorem{ex}[thm]{Example}

\newtheorem{prob}[thm]{Problem} 

\newtheorem{remark}[thm]{Remark}

\allowdisplaybreaks

\begin{document}
\title[Numerical semigroups from rational matrices {IV}]{Numerical semigroups from rational matrices {IV}: computation of the matricial dimensions of numerical semigroups with small Frobenius number or genus}

\author{Theo Chinn}
\address{Department of Mathematics and Statistics, Pomona College, 610 N. College Ave., Claremont, CA 91711, USA}
\email{twcz2023@mymail.pomona.edu}

\author{Junshu Feng}
\address{Department of Mathematics and Statistics, Pomona College, 610 N. College Ave., Claremont, CA 91711, USA}
\email{jfwv2022@mymail.pomona.edu}

\author{Stephan Ramon Garcia}
\address{Department of Mathematics and Statistics, Pomona College, 610 N. College Ave., Claremont, CA 91711, USA}
\email{stephan.garcia@pomona.edu}
\urladdr{\url{https://stephangarcia.sites.pomona.edu/}}

\thanks{SRG was partially supported by NSF grants DMS-2452084 and DMS-2054002.}

\author{Peiting Jiang}
\address{Department of Mathematics and Statistics, Pomona College, 610 N. College Ave., Claremont, CA 91711, USA}
\email{pjan2023@mymail.pomona.edu}

\keywords{semigroup; numerical semigroup; matricial dimension; companion matrix; integer linear programming.}

\begin{abstract}
We introduce a module-theoretic approach and a linear-programming method to compute the matricial dimension of numerical semigroups.  We use these to compute the matricial dimension of every numerical semigroup with Frobenius number at most $10$ or genus at most $6$.  Many of these evaluations were beyond the scope of previous techniques.
\end{abstract}

\maketitle

\section{Introduction}

In this paper, the term \emph{semigroup} refers to an additive subsemigroup of $\N = \{0,1,2,\ldots\}$.  A \emph{numerical semigroup} is a subsemigroup of $\N$ with finite complement \cite{Assi, Rosales}.  The notation $\semigroup{\cdot}$ denotes the semigroup generated by the indicated set.  For example, $\semigroup{4,7} = \{0, 4, 7, 8, 11, 12, 14, 15, 16, 18, 19, 20,\ldots\}$.  The \emph{Frobenius number} $F(S)$ of a numerical semigroup is the largest natural number not in $S$.  For example, $F(\semigroup{4,7}) = 17$. Each numerical semigroup has a unique minimal system of generators $n_1 < n_2 < \cdots < n_k$ such that $\gcd(n_1,n_2,\ldots,n_k) = 1$ and $S = \semigroup{n_1,n_2,\ldots,n_k}$ \cite[Thm.~2.7]{Rosales}.  The smallest generator $n_1$ is the \emph{multiplicity} $m(S)$ of $S$. The \emph{genus} $g(S)$ of $S$ is the cardinality of $\N \setminus S$.

Let $\M_d(\cdot)$ denote the set of $d \times d$ matrices with entries in the indicated set.  Let $\Z$ and $\Q$ denote the set of integers and rational numbers, respectively. The \emph{exponent semigroup} of $A \in \M_d(\Q)$ is 
\begin{equation*}
\S(A) = \{ n \in \N : A^n \in \M_d(\Z)\}.    
\end{equation*}
As its name suggests, the exponent semigroup of a rational matrix is a semigroup (not necessarily numerical) since it contains $0$ and is closed under addition.  Each additive subsemigroup of $\N$ is of the form $\S(A)$ for some $A \in \M_d(\Q)$ with $d \leq m(S)$ \cite[Thm.~1.1]{NSRM2}. In fact, the \emph{representing matrix} $A$ can be taken to be a weighted permutation matrix (the $1$s being replaced with suitable nonzero constants) for a cycle of length $d$.  This result improved the previous best-known bound $d \leq F(S) + 1$ from \cite[Thm.~6.2]{NSRM1}; in that case, the representing matrix can be taken to be a weighted nilpotent Jordan block of size $F(S) + 1$.

The \emph{matricial dimension} $\mdim S$ of a semigroup $S$ is the smallest dimension $d$ such that $S = \S(A)$ for some $A \in \M_d(\Q)$. Following \cite{NSRM3}, $S$ is \emph{dimension-$d$ realizable} if $S= \S(A)$ for some $A \in \M_d(\Q)$; this is equivalent to $\mdim S \leq d$ since $\S(A \oplus 0) = \S(A)$, in which $0$ denotes a square zero matrix.  One has the upper bound $\mdim S \leq m(S)$ \cite[Thm.~1.1]{NSRM2}, which is often attained.  However, $\mdim S$ is sometimes significantly smaller than $m(S)$. An open, and apparently difficult, problem is the determination of the matricial dimension of an arbitrary semigroup.  This paper takes several steps toward this goal.

This article is organized as follows. Section \ref{Section:Module} introduces a module-theoretic approach (Proposition \ref{Proposition:Module}) for finding representing matrices for semigroups.  This places several previous ad-hoc computations on a systematic footing while simultaneously permitting the computation of the matricial dimension for certain semigroups whose matricial dimension was previously unknown.  For example, we obtain a short module-theoretic proof (Theorem \ref{Theorem:NSRM2}) of a generalization of the main result of \cite{NSRM2}.  This provides dramatic improvements over previous matricial-dimension bounds for certain reducible semigroups (Theorem \ref{Theorem:Intersection}).  We collect several more ideas in Section \ref{Section:Multiplicity}, in which we determine the matricial dimension of all semigroups of multiplicity at most $3$. In Section \ref{Section:Companion}, we introduce a novel integer linear-programming approach, the ``companion-matrix method'', for obtaining representing matrices.  Taken together, our new results are sufficient to compute the matricial dimensions of all numerical semigroups with $F(S) \leq 10$ or $g(S) \leq 6$ (Section \ref{Section:Frobenius}).
We conclude in Section \ref{Section:Open} with several open questions spurred by our investigations.

\medskip\noindent\textbf{Acknowledgments.}
We thank Bjorn Poonen for several helpful suggestions.  In particular, his evaluation $\mdim\semigroup{5,7,8,9,11} = 2$, provided us with many new ideas.  We thank Christopher O'Neill for his help with Sage and GAP in the early stages of this project, and for several helpful comments on more recent drafts.  Special thanks go to Andrew Wilson of the Pomona College High Performance Computing (HPC) team.

\section{Module-theoretic approach}\label{Section:Module}

Our first theorem provides an abstract framework that relates exponent semigroups to $\Z$-modules in $\Q$-vector spaces.  This recovers several previous results (Theorem \ref{Theorem:LatticeExamples} parts (a) and (b)), generalizes a fundamental construction (Theorem \ref{Theorem:NSRM2}) that leads to vast improvements in certain cases (Theorem \ref{Theorem:Intersection}), and produces new results such as Theorem \ref{Theorem:LatticeExamples}(c).

\begin{prop}\label{Proposition:Module}
    Let $\V$ be a $d$-dimensional $\Q$-vector space containing $\MM \subseteq \V$, a $\Z$-module of rank $d$.  Let $L \in \operatorname{End}(\V)$ and let $\beta$ denote an integral basis for $\MM$.    
    Then $A = {}_{\beta}[L]_{\beta} \in \M_d(\Q)$ satisfies $\S(A) = \{n\in \N: L^n \MM \subseteq \MM \}$.
\end{prop}

\begin{proof}
Observe that
\begin{equation*}
    L^n \MM \subseteq \MM
    \iff {}_{\beta}[L^n]_{\beta} \in \M_d(\Z)
    \iff {}_{\beta}[L]_{\beta}^n \in \M_d(\Z) 
    \iff A^n \in \M_d(\Z).\qedhere    
\end{equation*}  
\end{proof}

The following is a generalization of \cite[Thm.~1.1]{NSRM2}. The module-theoretic proof below is completely different and avoids the use of the Kunz inequalities altogether. 

\begin{thm}\label{Theorem:NSRM2}
Let $S$ be a numerical semigroup.  If $d \in S \setminus \{0\}$, then there is an $A \in \M_d(\Q)$ such that $\S(A) = S$.  In particular, $\mdim S \leq m(S) = \min S \setminus \{0\}$.
\end{thm}

\begin{proof}
    Let $p$ be a prime and $\alpha =p^{1/d}$.  For $0 \leq i \leq d-1$, let $s_i$ be the smallest element of $S$ such that $s_i\equiv i\pmod{d}$; in particular, $s_0 = 0$. Then $\beta=\alpha^{s_0},\alpha^{s_1},\ldots,\alpha^{s_{d-1}}$ is an integral basis for the $\Z$-submodule $\MM = \vecspan_{\Z} \{\alpha^s:s\in S\}$ in $\V = \vecspan_{\Q}\{\alpha^s:s\in S\}$.   Let $L \in \operatorname{End}_{\Z}(\MM)$ be multiplication by $\alpha$. We claim that $S = \{n \in \N: L^n \MM \subseteq \MM\}$. 
    \begin{enumerate}[leftmargin=*]
        \item If $n \in S$, then $\alpha^n \MM = \vecspan_{\Z}\{ \alpha^{n+s} :s\in S\} \subseteq \MM$ since $n+s\in S$.

        \item If $n \notin S$, then $n<s_i$, in which $n \equiv i\pmod{d}$. Therefore, $\alpha^n \MM \subseteq \MM$ implies that $\alpha^{s_i} > \alpha^n = \alpha^n 1 = c \alpha^{s_i} > 0$ for some $c\in \Z$, which is impossible. 
    \end{enumerate}  
    We see that $s_i + 1 \equiv i +1 \equiv s_{i+1} \pmod{d}$; hence, $\alpha \alpha^{s_i} = \alpha^{s_i + 1} = p^{\ell_i}\alpha^{s_{i+1}}$ for some $\ell_i \in \Z$.  Proposition \ref{Proposition:Module} ensures that $A = {}_{\beta}[L]_{\beta}\in \M_d(\Q)$ satisfies $\S(A) = S$.
\end{proof}

\begin{ex}
    Let $S=\semigroup{6,9,20}$, the so-called McNugget semigroup \cite{MR3885968, MR4291927, MR4261927}, and let $d=6$. Let $p=2$ be a prime and $\alpha=2^{1/d}$.
    For $0\leq i\leq 5$, let $s_i$ be the smallest element of $S$ such that $s_i\equiv i\pmod{d}$: $s_0,s_1,s_2,s_3,s_4,s_5=0,49,20,9,40,29$, respectively. Then
    \begin{equation*}
    \beta= \alpha^{s_0},\alpha^{s_1},\ldots,\alpha^{s_{5}} 
    = \alpha^0, \alpha^{49}, \alpha^{20}, \alpha^9, \alpha^{40}, \alpha^{29}
    \end{equation*}
    is an integral basis for the $\Z$-module $\MM = \vecspan_{\Z} \{\alpha^s:s\in S\}$. 
    Since
    \begin{align*}
        \alpha\alpha^{s_0} &=\alpha\alpha^0=\alpha^1=\alpha^{-48}\alpha^{49} = 2^{-8} \alpha^{49} = \tfrac{1}{256} \alpha^{s_1},\\
        \alpha\alpha^{s_1} &=\alpha\alpha^{49}=\alpha^{50}=\alpha^{30}\alpha^{20} = 2^5 \alpha^{20} = 32 \alpha^{s_2}, \\
        \alpha\alpha^{s_2} &=\alpha\alpha^{20}=\alpha^{21}=\alpha^{12}\alpha^{9} = 2^2 \alpha^9  = 4 \alpha^{s_3}, \\
        \alpha\alpha^{s_3} &=\alpha\alpha^{9}=\alpha^{10}=\alpha^{-30}\alpha^{40} = 2^{-5} \alpha^{40} = \tfrac{1}{32}\alpha^{s_4}, \\
        \alpha\alpha^{s_4} &=\alpha\alpha^{40}=\alpha^{41}=\alpha^{12}\alpha^{29} = 2^{2} \alpha^{29} = 4 \alpha^{s_5}, \\
        \alpha\alpha^{s_5} &=\alpha\alpha^{29}=\alpha^{30}=\alpha^{30}\alpha^{0} = 2^{5} \alpha^0 = 32\alpha^{s_0},        
    \end{align*}
    Theorem \ref{Theorem:NSRM2} ensures that $\S(A) = \semigroup{6,9,20}$ for
    \begin{equation*}
        A = {}_{\beta}[L]_{\beta}
        = 
        \begin{bmatrix}
        \0&\0&\0&\0&\0&32\\
        \frac{1}{256}&\0&\0&\0&\0&\0\\
        \0&32&\0&\0&\0&\0\\
        \0&\0&4&\0&\0&\0\\
        \0&\0&\0&\frac{1}{32}&\0&\0\\
        \0&\0&\0&\0&4&\0
        \end{bmatrix}\in \M_6(\Q).
    \end{equation*}
    This is the transpose of the weighted permutation matrix obtained in \cite[~Ex. 3.1]{NSRM2}.
\end{ex}

\begin{remark}
One can adapt the construction of \cite[Thm.~1.1]{NSRM2} to prove Theorem \ref{Theorem:NSRM2} as follows: replace the Ap\'ery set $\Ap(S, m(S))$ with $\Ap(S,d)$ and use \cite[Lem.~1.4]{Rosales} to obtain the Kunz inequalities \cite[(2.1)]{NSRM2} directly.  However, the module-theoretic approach is superior since it is more general and does not require the Kunz inequalities.
\end{remark}

An upper bound on the matricial dimension can sometimes be obtained by representing a numerical semigroup as an intersection of other numerical semigroups.  Recall that the best-known general upper bound on the matricial dimension is $\mdim S \leq m(S)$ \cite[Thm. ~1.1]{NSRM2}.  Our next theorem often improves upon this.

\begin{thm}\label{Theorem:Intersection}
Let $S_1,S_2,\ldots,S_k$ be numerical semigroups and let $S = \bigcap_{i=1}^k S_i$.  Suppose that $d_i \in S_i \setminus \{0\}$ for $i=1,2,\ldots,k$.  Then there exists an $A\in \M_{d_1+d_2+\cdots + d_k}(\Q)$ such that $\S(A) = S$.  In particular, $\mdim S \leq d_1 + d_2 + \cdots + d_k$.
\end{thm}

\begin{proof}
First, recall that the finite intersection of numerical semigroups is a numerical semigroup \cite[Exercise 2.2]{Rosales}.  For $i=1,2,\ldots,k$, Theorem \ref{Theorem:NSRM2} provides an $A_i \in \M_{d_i}(\Q)$ such that $\S(A_i) = S_i$.  Thus,
$S=\S(A)$ for $A = A_1 \oplus A_2 \oplus \cdots \oplus A_k \in \M_{d_1+d_2+\cdots + d_k}(\Q)$.
\end{proof}

If $d_1+d_2+\cdots+d_k<m(S)$, then Theorem \ref{Theorem:Intersection} improves upon \cite[Thm. ~1.1]{NSRM2}, which established $\mdim S \leq m(S)$.  This can actually occur, as the next example demonstrates.

\begin{ex}
Consider $S = \semigroup{30,31,36,40,45,46,51,55,65}$. Before this paper, the best we could say was $\mdim S \leq m(S) = 30$ \cite[Thm. ~1.1]{NSRM2}; this handily beats the previous bound $\mdim S \leq 120 = F(S) + 1$ from \cite[Thm.~6.2]{NSRM1}.  Since
    \begin{equation*}    
        S=  \semigroup{2,31} 
        \cap \semigroup{3,31}\cap \semigroup{5,31},
    \end{equation*}
    Theorem \ref{Theorem:Intersection} ensures that $\mdim S \leq 2+3+5 = 10$, a three-fold improvement.
\end{ex}

\begin{ex}
Let $p,q,r$ be distinct primes such that $r > pq$ and consider $S = \semigroup{p,r} \cap \semigroup{q,r}$.
Then $\mdim S \leq \mdim \semigroup{p,r} + \mdim \semigroup{q,r} = p+q$ by Theorem \ref{Theorem:Intersection}.  In relative terms, this estimate can be arbitrarily better than the upper bound $\mdim S \leq m(S)= pq$ provided by Theorem \ref{Theorem:NSRM2} since
$(p+q)/pq  = 1/p + 1/q$ tends to $0$ as $p,q \to \infty$.
\end{ex}

Parts (a) and (b) of the next theorem are contained in \cite[Prop.~2.2]{NSRM1} (the matrices involved are different, however). In contrast with the ad-hoc method of \cite{NSRM1}, the module-theoretic approach, based upon Proposition \ref{Proposition:Module}, is systematic.  Part (c) is novel and it is the most important result in our computation of matricial dimensions for semigroups of low complexity; see Theorem \ref{Theorem:m3} and Section \ref{Section:Frobenius}.

\begin{thm}\label{Theorem:LatticeExamples}
The following semigroups have matricial dimension $2$.
    \begin{enumerate}[leftmargin=*]
        \item Cyclic semigroups: $\semigroup{a}$ with $a \geq 2$.
        \item Ordinary semigroups: $S_b = \{b,b+1,b+2,\ldots\}$ with $b \geq 2$.
        \item The union of a cyclic and ordinary semigroup: $\semigroup{a} \cup S_b$ with $a,b\geq 2$.
    \end{enumerate}
\end{thm}

\begin{proof}
We sketch the constructions in (a) and (b) while presenting (c) in full detail.  This is no loss, since (c) is more complex and involves aspects of both previous cases.  It is more convenient here to work with matrices instead of transformations.  In the terminology of Proposition \ref{Proposition:Module}, $\V = \Q^2$, $\beta = \{\vec{b}_1, \vec{b}_2\}$, in which $B = [\vec{b}_1~\vec{b}_2] \in \M_2(\Q)$, and $\MM = B\Z^2$.

\medskip\noindent(a) Fix $a \geq 2$, let $p \equiv 1 \pmod{a}$ be a prime, and let $\ord_p z = a$.
    Let $B = \tinymatrix{1}{1/p}{\0}{1/p}$ and $L = \tinymatrix{z}{\0}{\0}{1}$. Then $A = B^{-1}LB = \tinymatrix{z}{(z-1)/p}{\0}{1}$, so  $A^n = \tinymatrix{z^n}{(z^n-1)/p}{\0}{1}$ and $\S(A) = \semigroup{a}$.

\medskip\noindent(b) Fix $m\geq 1$, let $p$ be a prime, let $B = \tinymatrix{1}{1/p^b}{\0}{1/p^b}$, and let $L = \tinymatrix{p^2}{\0}{\0}{p}$.  Then $A = B^{-1}LB= \tinymatrix{p^2}{(p-1)p^{1-m}}{\0}{p}$, so  $A^n = \tinymatrix{p^{2n}}{p^{n-b}(p^n-1)}{\0}{p^n}$ and $\S(A) = S_b$.

\medskip\noindent(c)     Fix $a,b \geq 2$, let $p \equiv 1 \pmod{a}$ be prime, and let $\ord_{p^b} z = a$.  Let
    $B = [ \vec{b}_1~\vec{b}_2] = \tinymatrix{1}{1/p^b}{\0}{1/p^b}$ and $L = \tinymatrix{zp}{\0}{\0}{p}$.  Then $A = B^{-1}LB = \tinymatrix{pz}{p^{1-b}(z-1)}{\0}{p}$, so 
    \begin{equation*}
    A^n= \begin{bmatrix} p^n z^n & p^{n-b}(z^n-1) \\ \0 & p^n \end{bmatrix}.    
    \end{equation*}
    We claim that $\S(A) = \semigroup{a} \cup S_b$.  Since $\vec{b}_1 \in \Z^2$, we only need to determine when 
    \begin{equation*}
     L^n \vec{b}_2 = \begin{bmatrix} p^{n-b}z^n \\ p^{n-b} \end{bmatrix} \in \MM = \vecspan_{\Z}\{ \vec{b}_1 , \vec{b}_2\} = B\Z^2.   
    \end{equation*}
    This occurs whenever there exists an $\vec{x}=\begin{bmatrix} x_1 \\ x_2 \end{bmatrix} \in \Z^2$ such that $L^n \vec{b}_2 = x_1 \vec{b}_1 + x_2 \vec{b}_2$.  Solve the corresponding system
    $p^{n-b} z^n = x_1 + p^{-b} x_2$ and $p^{n-b} = 0 x_1 + p^{-b} x_2$,
    and get $x_2 = p^n$, so that $x_1 = p^{n-b}(z^n - 1)$, which is an integer if and only if
    $n \geq b$ or $z^n \equiv 1 \pmod{p^b}$.  The second condition is equivalent to $a \mid n$.
\end{proof}


Prior to this paper, even relatively simple numerical semigroups
such as $\semigroup{4,6,7,9}$ and $\semigroup{5,7,8,9,11}$ were beyond our ability to handle.  The next two examples use Theorem \ref{Theorem:LatticeExamples}(c) to compute the matricial dimension of these previously intractable semigroups.

\begin{ex}\label{Ex:4679}
    The matricial dimension of $S= \semigroup{4} \cup S_6 = \semigroup{4,6,7,9}$ is $2$. Let $a=4$ and $b=6$ in the proof of Theorem \ref{Theorem:LatticeExamples}(c). Select $p=5 \equiv 1 \pmod{4}$ and observe that $\ord_{5^6} z = 4$, in which $z = 14557$. Then $\S(A)=S$ for $A = \tinymatrix{pz}{p^{1-b}(z-1)}{\0}{p}=\tinymatrix{72785}{\frac{14556}{3125}}{\0}{5}$.
\end{ex}

\begin{ex}\label{Ex:578911}
    The matricial dimension of $S = \semigroup{5} \cup S_7 = \semigroup{5,7,8,9,11}$ is $2$. Let $a=5$ and $b=7$ in the proof of Theorem \ref{Theorem:LatticeExamples}(c). Select $p=11 \equiv 1 \pmod{5}$ and observe that $\ord_{11^7} 10{,}250{,}663 = 5$.  Then $\S(A) = S$ for $A=\tinymatrix{112{,}757{,}293}{\frac{10{,}250{,}662}{1{,}771{,}561}}{\0}{11}$.
\end{ex}

\section{Numerical semigroups of low multiplicity}\label{Section:Multiplicity}

In this section, we compute the matricial dimension of numerical semigroups of multiplicity at most $3$.  If $m(S) = 1$, then $S = \N$ and hence $\mdim S = 1$ \cite[Prop.~2.1]{NSRM1}.  If $m(S) = 2$, then $S = \semigroup{2}$ or $S = \semigroup{2,k}$ with $k\geq 3$ odd; in both cases, the semigroups have matricial dimension $2$ (Theorem \ref{Theorem:LatticeExamples}).

For $m(S) = 3$, we require new ideas.  A nonzero element of a numerical semigroup $S$ is \emph{small} if it is less than the Frobenius number, $F(S)$.  The set of nonzero small elements is 
\begin{equation*}
\s(S) = \{ n \in S \setminus \{0\} : n < F(S)\}.
\end{equation*}
The following lemma is \cite[Thm.~7.2]{NSRM3}.

\begin{lemma}\label{Lemma:Small}
    If $S$ is dimension-$2$ realizable, then $\s(S) = \varnothing$ or $\gcd\s(S) \geq 2$.
\end{lemma}

Arsh Chhabra (personal communication) noted the following useful extension, which will be useful at a certain point in Section \ref{Section:Frobenius}, although it is tangential to our present aim.

\begin{thm}\label{Theorem:Quotient}
    Let $S$ be dimension-$d$ realizable and $g = \gcd \s(S) \geq 2$.
    Then $S/g = \{ n \in \N : gn \in S\}$ is dimension-$d$ realizable.
\end{thm}

\begin{proof}
    If $S = \S(A)$ for some $A \in \M_d(\Q)$, let $B = A^g$ and observe that $\S(B) = \{ n \in \N : B^n \in \M_d(\Z)\} = \{n \in \N : A^{gn} \in \M_d(\Z)\} = \{ n \in \N : gn \in S\} = S/g$.
\end{proof}

We first need the following rephrasing of \cite[Thm.~5.1]{NSRM1} in terms of matricial dimension.  It is the primary method for establishing lower bounds on the matricial dimension.

\begin{lemma}\label{Lemma:Consecutive}
    If $S$ is a numerical semigroup such that $n,n+1,n+2,\ldots,n+r-1 \in S$ and $n+r \notin S$,
    then $\mdim S > r$.
\end{lemma}

Here is the main theorem of this section.

    \begin{thm}\label{Theorem:m3}
    Let $S$ be a numerical semigroup with $m(S) = 3$ and $\Ap(S)=\{0,a_1,a_2\}$,
    in which $a_i \equiv i \pmod{3}$ for $i=1,2$. 
    \begin{enumerate}[leftmargin=*]
        \item If $|a_1-a_2|\leq 2$, then $\mdim S=2$.
        \item If $|a_1-a_2| > 2$, then $\mdim S=3$.
    \end{enumerate}
\end{thm}

\begin{proof}
Since $m(S) = 3$, Theorem \ref{Theorem:NSRM2} ensures that $\mdim S \leq 3$.

\medskip\noindent(a) If $|a_1 - a_2| \leq 2$, then $S = \semigroup{3} \cup \S_{\min\{a_1,a_2\}}$, so
$\mdim S = 2$ by Theorem \ref{Theorem:LatticeExamples}.

\medskip\noindent(b) If $|a_1 - a_2| > 2$, then $\{3,\min\{a_1,a_2\}\}\subseteq\s(S)$ and $\gcd\s(S)=1$. 
Lemma \ref{Lemma:Small} ensures that $S$ is not dimension-$2$ realizable, so $\mdim S=3$.
\end{proof}


\section{The companion-matrix method}\label{Section:Companion}

This paper culminates in the computation of the matricial dimensions of all numerical semigroups $S$ with $F(S) \leq 10$ or $g(S) \leq 6$.  This would be impossible without the introduction of novel methods, since many of these semigroups were beyond the means of \cite{NSRM1, NSRM2, NSRM3}.  We explain here a powerful new method for finding representing matrices of size smaller than $m(S)$.  Since there is a bit of intuition and a lot of linear programming involved, it is difficult to state our approach as a formal theorem.  Consequently, we describe the technique as we have employed it so that others can apply the method.

The method illustrated in the following examples is called the \emph{companion-matrix method} (CMM).
It relies on the fact that the characteristic polynomial $p_A(x)$ of any $A \in \M_d(\Q)$ with $\S(A) \neq \{0\}$
has integer coefficients \cite[Thm.~3.3(a)]{NSRM1}.  Given a semigroup $S$, one can often make an educated guess about the form of the characteristic polynomial of a generating matrix for $S$ and from there attack the problem with integer linear programming.

\begin{ex}\label{Example:457}
Consider $S = \semigroup{4,5,7}$, which has $F(S) = 6$. Since $4,5 \in S$ and $6 \notin S$, Lemma \ref{Lemma:Consecutive} ensures that $\mdim S \geq 3$.  Since $m(S) = 4$, Theorem \ref{Theorem:NSRM2} implies that $\mdim S \leq 4$.  Therefore, $\mdim S \in \{3,4\}$.  We claim that $\mdim S = 3$, and we prove this by constructing an $A \in \M_3(\Q)$ such that $\S(A) = S$.  

If $a\in \Z$ and $p_A(x)$ divides $x^4 - a^4 = (x-a) (x+a) (x^2+a^2)$ in $\Z[x]$, the Cayley--Hamilton theorem ensures that $A^{4n} = a^{4n}I \in \M_3(\Z)$; that is, $\semigroup{4} \subseteq \S(A)$.  We search for representing matrices among rational matrices whose characteristic polynomial divides $x^4-a^4$. At a later stage, we require $a \neq \pm 1$.  For the sake of simplicity, we choose $a=2$.

The characteristic polynomial of the desired matrix $A \in \M_3(\Q)$ (if it exists) might be a cubic divisor of $x^4 - 2^4 = (x-2)(x+2)(x^2+4)$.  We choose $f(x) = (x-2)(x^2+4) = x^3-2 x^2+4 x-8$ and let
\begin{equation*}
    C_f = 
    \begin{bmatrix}
            \0 & \0 & 8 \\
            1 & \0 & -4\\
            \0 & 1 & 2
    \end{bmatrix}
\end{equation*}
denote the corresponding companion matrix, so $f(C_f) = 0$.  If $A = B^{-1} C_f B$, then $p_A(x) = f(x)$, so $\semigroup{4} \subseteq \S(A)$.  We seek a $B \in \M_3(\Q)$ such that $S = \semigroup{4,5,7} = \S(A)$.

Let $B = \operatorname{diag}(2^{x_1},2^{x_2},2^{x_3})$, in which $x_1,x_2,x_3 \in \Z$ are to be determined.  Then
\begin{equation*}
    A = B^{-1}C_p B = 
\begin{bmatrix}
     \0 & \0 & 2^{-x_1+x_3+3} \\
 2^{x_1-x_2} & \0 & -2^{-x_2+x_3+2} \\
 \0 & 2^{x_2-x_3} & 2 \\
\end{bmatrix}.
\end{equation*}
For $\ell \in \N$, observe that $A^{4\ell} = 16^{\ell} I$,
\begin{equation*}
    \begin{split}
    A^{4\ell+1}
    &=
\left[
\begin{array}{ccc}
 \0 & \0 & 2^{4 \ell-x_1+x_3+3} \\
 2^{4 \ell+x_1-x_2} & \0 & -2^{4 \ell-x_2+x_3+2} \\
 \0 & 2^{4 \ell+x_2-x_3} & 2^{4 \ell+1} \\
\end{array}
\right], \\[2pt]
A^{4\ell+2}
&= 
\left[
\begin{array}{ccc}
 \0 & 2^{4 \ell-x_1+x_2+3} & 2^{4 \ell-x_1+x_3+4} \\
 \0 & -4^{2 \ell+1} & \0 \\
 2^{4 \ell+x_1-x_3} & 2^{4 \ell+x_2-x_3+1} & \0 \\
\end{array}
\right], \quad \text{and} \\[2pt]
A^{4\ell+3}
&=
\left[
\begin{array}{ccc}
 2^{4 \ell+3} & 2^{4 \ell-x_1+x_2+4} & \0 \\
 -2^{4 \ell+x_1-x_2+2} & \0 & 2^{4 \ell-x_2+x_3+4} \\
 2^{4 \ell+x_1-x_3+1} & \0 & \0 \\
\end{array}
\right].
\end{split}
\end{equation*}
Since each exponent is a linear expression in the integer variables $x_1,x_2,x_3$, our search for a suitable $B$ can be reduced to an integer linear-programming problem.

We seek $x_1,x_2,x_3 \in \Z$ such that $A^4,A^5,A^7 \in \M_3(\Z)$ and such that $A,A^2,A^3,A^6 \notin \M_3(\Z)$.
Since $F(S) = 6$ and $S = \semigroup{4,5,7}$, it would follow that $\S(A) = S$.
Because $A^4 = 16 I \in \M_3(\Z)$, we can ignore this constraint.  
The condition $A^5 \in \M_3(\Z)$ requires that each of the following linear inequalities must be satisfied:
\begin{equation}\label{eq:LP1}
    \begin{split}
        4+x_1 - x_2 \geq 0,\\
        4 + x_2 - x_3 \geq 0,\\
        6 - x_2 + x_3 \geq 0,\\
        7-x_1 + x_3 \geq 0.
    \end{split}
\end{equation}
We pick up another system of linear inequalities for the condition $A^7 \in \M_3(\Z)$:
\begin{equation}\label{eq:LP2}
    \begin{split}
        5 + x_1 - x_3 \geq 0,\\
        6 + x_1 - x_2 \geq 0,\\
        8 - x_1 + x_2 \geq 0,\\
        8 - x_2 + x_3 \geq 0.
    \end{split}
\end{equation}
Since we want $A \notin \M_3(\Z)$, at least one of the following must hold:
\begin{equation*}
        x_1 -x_2 \leq - 1,\quad
        x_2 - x_3 \leq - 1,\quad
        -x_1+x_2+3 \leq - 1,\quad \text{and} \quad
        -x_2 + x_3 + 2 \leq - 1.    
\end{equation*}
We select $x_1 - x_2 \leq - 1$ and append it to \eqref{eq:LP1} and \eqref{eq:LP2}.  Similarly, the condition
$A^3 \notin \M_3(\Z)$ yields (among other possibilities) the inequality $2+x_1 - x_2 \leq -1$.
From $A^6 \notin \M_3(\Z)$ we obtain $4 + x_1 - x_3 \leq -1$.  So  in total we append to \eqref{eq:LP1} and \eqref{eq:LP2} 
the inequalities
\begin{equation}\label{eq:LP3}
    \begin{split}
        x_1 - x_2 \leq - 1, \\
        2+x_1 - x_2 \leq -1,\\
        4 + x_1 - x_3 \leq -1.
    \end{split}
\end{equation}
Of the many solutions to \eqref{eq:LP1}, \eqref{eq:LP2}, and \eqref{eq:LP3} in $\Z^3$ we choose $(x_1,x_2,x_3) = (0,3,5)$,
so $B = \operatorname{diag}(1,8,32)$ and
\begin{equation*}
    A = 
    \begin{bmatrix}
     \0 & \0 & 256 \\
     \frac{1}{8} & \0 & -16 \\
     \0 & \frac{1}{4} & 2 \\
    \end{bmatrix},
\end{equation*}
which satisfies $\S(A) = S = \semigroup{4,5,7}$, as desired.
\end{ex}

While there is no guarantee that this \emph{companion-matrix method} works, it often does in practice.  \texttt{Mathematica} or any other comparable software can solve these small-scale integer-linear programming problems instantly or, at worst, declare that no solutions exists.  In the latter case, one can try again with a different potential characteristic polynomial.

\begin{ex}\label{Example:561314}
    Consider $S = \semigroup{5,6,13,14}$, which has $\mdim S \geq 3$ by Lemma \ref{Lemma:Consecutive}. Since $m(S)=5$ and the only nonlinear factor of $x^5-a^5$ over $\Z[x]$ is a single degree-$4$ factor, the CMM cannot easily produce a $3 \times 3$ representing matrix for $S$. One way to use the CMM for semigroups with $m(S)=5$ is to use the characteristic polynomial of a representing matrix for a compatible model semigroup. To illustrate this, we show that $\S(A)=S$ for 
    \begin{equation}\label{eq:ABCBpAbove}
    A = 
    \begin{bmatrix}
            \0 & \0 & \frac{1}{262144} \\[2pt]
            4096 & \0 & -\frac{1}{128}\\[2pt]
            \0 & 4096 & 8
    \end{bmatrix}
\end{equation} 
and thus $\mdim S = 3$. We want $S$ and the model semigroup to share some elements so that the CMM has the best chance of producing $S$; this is analogous to how we chose the characteristic polynomial in Example \ref{Example:457} so that $\semigroup{4} \subseteq \S(A)$ by construction. Take $S' = \semigroup{6,7,11,15,16}$, which shares $6,11,12,13,14,15,16,\ldots$ with $S$. Lemma \ref{Lemma:Consecutive} ensures that $\mdim S' \geq 3$ and one can show, as we do below, that a representing matrix for $S'$ is 
\begin{equation*}
    A'=
\begin{bmatrix}
 \0 & \0 & \frac{1}{512} \\[2pt]
 16777216 & \0 & -16384 \\[2pt]
 \0 & \frac{1}{512} & 8 \\
\end{bmatrix},
\end{equation*} 
so $\mdim S' = 3$. Since $m(S') = 6$ and 
\begin{equation*}
x^6 - a^6 = (x-a) (x+a) (x^2-a x +a^2) (x^2+a x+a^2)    
\end{equation*}
has four degree-$3$ factors, it is much easier to find a $3 \times 3$ representing matrix for $S'$ than for $S$.
The characteristic polynomial of $A'$ is $p_{A'}(x) = x^3 - 8x^2 + 32x - 64$. We take this as our educated guess
for a possible characteristic polynomial for the desired $A$. Using 
\begin{equation*}
    C_{p_{A'}(x)} = 
    \begin{bmatrix}
            \0 & \0 & 64 \\
            1 & \0 & -32\\
            \0 & 1 & 8
    \end{bmatrix}
    \quad \text{and} \quad
    B = \operatorname{diag}(2^{12},2^{0},2^{-12}),
\end{equation*} 
in which $B$ was obtained through integer linear programming,
we obtain the matrix $A = B^{-1}C_{p_{A'}}B$ given in \eqref{eq:ABCBpAbove}.
\end{ex}

\begin{ex}
    As Example \ref{Example:561314} suggests, semigroups with $m(S) = 5$ can be difficult to find $3 \times 3$ generating matrices for (if they exist at all).  For example, $S=\semigroup{5,6,8}$ proved the most stubborn of all, and it required new techniques altogether.  First observe that Lemma \ref{Lemma:Consecutive} ensures that $\mdim S \geq 3$, so we are in the same difficult position as we were in Example \ref{Example:561314}.  Despite many attempts, the method of that example appeared inadequate.
    However, we could take an educated guess as to the form of $p_A(x)$ for an $A \in \M_3(\Q)$ such that $\S(A) = S$.
    Suppose that $\S(A) = \semigroup{5,6,8}$, in which $A \in \M_3(\Q)$ and $p_A(x) = x^3 + px^2 +qx + r$ with $p,q,r \in \Z$.
    The Cayley--Hamilton theorem ensures that $A^3 + pA^2 + qA + rI = 0$.  From this we can make a variety of observations.
    \begin{enumerate}[leftmargin=*]
        \item First note that $r\neq 0$ (so $A$ is invertible). Suppose toward a contradiction that $r=0$. Then $A^7 + pA^6 + qA^5 = A^4(A^3 + pA^2 + qA + rI) = 0$, so $A^7 \in \M_3(\Z)$ since $5,6 \in S$.  Since this contradicts the fact that $7 \notin S$, we conclude that $r \neq 0$.  

        \item Next observe that $q \neq 0$.  Suppose toward a contradiction that $q = 0$.  Then
$A^9 + pA^8  + r A^6 = A^6 (A^3 + pA^2 + qA + rI) = 0$, so $A^9 \in \M_3(\Z)$
since $6,8 \in S$.  Since this contradicts the fact that $9 \notin S$, we conclude that $q \neq 0$.

\item We claim that $\gcd(q,r) \geq 2$.  Since $A^{11} + pA^{10} + qA^9 + r A^8 = A^8 p_A(A) =0$ and $8,10,11\in S$,
we conclude that $qA^9\in \M_3(\Z)$.  In a similar manner, examine $A^{12} + pA^{11} + qA^{10} + r A^9 = 0$ and deduce that
$rA^9 \in \M_3(\Z)$.  Therefore, $\gcd(q,r) A^9 \in \M_3(\Z)$.  Since $9 \notin S$, we conclude that $\gcd(q,r) \geq 2$.
    \end{enumerate}
    Similar arguments suggest that $\gcd(p,q,r)\geq 2$.
    Preliminary computations indicated that $2$ did not provide enough granularity.    
    Ultimately, we took $\gcd(p,q,r) = 8$ and used the CMM as follows.  We
    generated random polynomials $f(x)=x^3+px^2+qx+r$ with $\gcd(p,q,r)=8$ and random lower-triangular matrices $B$ whose nonzero entries were nonnegative powers of two.  Since inversion of such $B$ is fast, we could test matrices of the form $A = B^{-1} C_f B$ by the billions: compute $A^n = B^{-1} C_f^n B$ for $n=1,2,3,\ldots,14$, checking step-by-step whether $A^n \in \M_3(\Z)$ for $n\in S = \{5,6,8,10,11,12,13,14,\ldots\}$ (since $5 \in S$, Lemma \ref{Lemma:Consecutive} ensures that we need check no power higher than $14$).  At the first disagreement, reject $A$ and start again.  
    On a high-performance computing cluster, this method quickly produced the generating matrix
    \begin{equation*}
        A =  B^{-1} C_f B =
        \begin{bmatrix}    
         -12 & -196608 & -768 \\[2pt]
         \frac{2363}{4096} & 9448 & \frac{1181}{32} \\[2pt]
         -\frac{273}{4} & -1379072 & -5388 \\
        \end{bmatrix},
    \end{equation*}
    in which $f(x) = x^3 - 4048 x^2 + 2840 x + 48$ and 
    \begin{equation*}
        B = 
        \begin{bmatrix}
         2^2 & \0 & \0 \\
         2^{10} & 2^{14} & \0 \\
         1 & 2^{14} & 2^6 \\   
        \end{bmatrix}.
    \end{equation*}   
\end{ex}

\section{Matricial dimension by Frobenius number and genus}\label{Section:Frobenius}
We now demonstrate the collective power of our techniques by determining the matricial dimension of every numerical semigroup $S$ with Frobenius number $F(S) \leq 10$ or genus $g(S) \leq 6$.  Semigroups with Frobenius number at most $4$ can be attacked with methods from \cite{NSRM1}. As the Frobenius number increases, methods from this paper are required.  Theorems \ref{Theorem:NSRM2}, \ref{Theorem:LatticeExamples}, and \ref{Theorem:m3}, and the companion-matrix method (CMM) descibed in Section \ref{Section:Companion}, become the primary workhorses. Ultimately, high-performance computing becomes valuable for addressing the final few difficult semigroups.

\bigskip\noindent\textbf{Matricial dimension by Frobenius number.}

{\small
\bigskip\noindent\textbf{$F(S)=1$}
\begin{itemize}\addtolength{\itemsep}{2pt plus 1pt minus 0.5pt}
\item $S=\semigroup{2,3}$: $\mdim S=2$ by Theorem \ref{Theorem:LatticeExamples}(c) with $a=2$ and $b=3$.
\end{itemize}

\bigskip\noindent\textbf{$F(S)=2$}
\begin{itemize}\addtolength{\itemsep}{2pt plus 1pt minus 0.5pt}
\item $S = \semigroup{3,4,5}$: $\mdim S = 2$ by Theorem \ref{Theorem:m3}.
\end{itemize}

\bigskip\noindent\textbf{$F(S)=3$}
\begin{itemize}\addtolength{\itemsep}{2pt plus 1pt minus 0.5pt}
    \item $S=\semigroup{2,5}$: $\mdim S=2$ by Theorem \ref{Theorem:LatticeExamples}(c) with $a=2$ and $b=5$.

    \item $S=\semigroup{4,5,6,7}$: $\mdim S=2$ by Theorem \ref{Theorem:LatticeExamples}(b) since $S = S_4$.
\end{itemize}

\bigskip\noindent\textbf{$F(S)=4$}
\begin{itemize}\addtolength{\itemsep}{2pt plus 1pt minus 0.5pt}
    \item $S=\semigroup{3,5,7}$: $\mdim S =2$ by Theorem \ref{Theorem:m3} (see also \cite[Rem.~5.8]{NSRM1}).
    
    \item $S=\semigroup{5,6,7,8,9}$: $\mdim S =2$ by Theorem \ref{Theorem:LatticeExamples}(b) since $S = S_5$.
\end{itemize}

\bigskip\noindent\textbf{$F(S)=5$}
\begin{itemize}\addtolength{\itemsep}{2pt plus 1pt minus 0.5pt}
    \item $S=\semigroup{2,7}$: $\mdim S=2$ by Theorem \ref{Theorem:LatticeExamples}(c) with $a=2$ and $b=7$.
    
    \item $S=\semigroup{3,4}$: $\mdim S =3$ by Theorem \ref{Theorem:m3} (see also \cite[Ex.~5.2]{NSRM1}). 

    \item $S=\semigroup{3,7,8}$: $\mdim S =2$ by Theorem \ref{Theorem:m3}.

    \item $S=\semigroup{4,6,7,9}$: $\mdim S =2$ by Theorem \ref{Theorem:LatticeExamples}(c) with $a=4$ and $b=6$.
    
    \item $S=\semigroup{6,7,8,9,10,11}$: $\mdim S =2$ by Theorem \ref{Theorem:LatticeExamples}(b) since $S = S_6$.
\end{itemize}

\bigskip\noindent\textbf{$F(S)=6$}
\begin{itemize}\addtolength{\itemsep}{2pt plus 1pt minus 0.5pt}
    \item $S=\semigroup{4,5,7}$: $\mdim S =3$.  See Example \ref{Example:457}.

    \item $S=\semigroup{4,7,9,10}$: $\mdim S =2$ by Theorem \ref{Theorem:LatticeExamples}(c) with $a=4$ and $b=7$.
    
    \item $S=\semigroup{5,7,8,9,11}$: $\mdim S =2$ by Theorem \ref{Theorem:LatticeExamples}(c) with $a=5$ and $b=7$.
    
    \item $S=\langle7,8,9,10,11,12,13 \rangle$: $\mdim S =2$ by Theorem \ref{Theorem:LatticeExamples}(b) since $S = S_7$.
\end{itemize}

\bigskip\noindent\textbf{$F(S)=7$}
\begin{itemize}\addtolength{\itemsep}{2pt plus 1pt minus 0.5pt}
    \item $S=\semigroup{2,9}$: $\mdim S=2$. Use Theorem \ref{Theorem:LatticeExamples}(c) with $a=2$ and $b=9$.
    
    \item $S=\semigroup{3,5}$: $\mdim S =3$ by Theorem \ref{Theorem:m3}.
    
    \item $S=\semigroup{3,8,10}$: $\mdim S =2$ by Theorem \ref{Theorem:m3}.

    \item $S=\semigroup{4,5,6}$: $\mdim S =4$.  Since $4,5,6 \in S$ and $7 \notin S$, Lemma \ref{Lemma:Consecutive} ensures that $\mdim S \geq 4$.  Theorem \ref{Theorem:NSRM2} with $d=4$ shows that $\mdim S= 4$.
    
    \item $S=\semigroup{4,5,11}$: $\mdim S =3$.  Since $4,5 \in S$ and $6 \notin S$, Lemma \ref{Lemma:Consecutive} ensures that $\mdim S \geq 3$.  The CMM ensures that $\mdim S = 3$ since $S = \S\left( \left[\begin{smallmatrix}\0 & \0 & \frac{1}{8} \\512 & \0 & -32 \\ \0 & \frac{1}{8} & 2 \\\end{smallmatrix}\right]\right)$.
    
   \item $S=\semigroup{4,6,9,11}$: $\mdim S =2$.  
The CMM provides $S = \S\left(\left[\begin{smallmatrix}0&32\\ \frac{1}{16} & 8\end{smallmatrix}\right]\right)$.
    
    \item $S=\semigroup{4,9,10,11}$: $\mdim S =2$. Use Theorem \ref{Theorem:LatticeExamples}(c) with $a=4$ and $b=9$.
    
    \item $S=\semigroup{5,6,8,9}$: $\mdim S =3$.  Since $5,6 \in S$ and $7 \notin S$, Lemma \ref{Lemma:Consecutive} ensures that $\mdim S \geq 3$. The CMM ensures that $\mdim S = 3$ since
    $S = \S\left( \left[\begin{smallmatrix} \0 & \0 & 16 \\ 64 & \0 & -1024 \\ \0 & \frac{1}{128} & 4 \end{smallmatrix}\right]\right)$.

   \item $S=\semigroup{5,8,9,11,12}$:
    $\mdim S =2$.  Use Theorem \ref{Theorem:LatticeExamples}(c) with $a=5$ and $b=8$.

    \item $S=\semigroup{6,8,9,10,11,13}$: $\mdim S =2$. Use Theorem \ref{Theorem:LatticeExamples}(c) with $a=7$ and $b=8$.
    
    \item $S=\semigroup{8,9,10,11,12,13,14,15}$: $\mdim S =2$ by Theorem \ref{Theorem:LatticeExamples}(b) since $S = S_8$.
\end{itemize}

\bigskip\noindent\textbf{$F(S)=8$}
\begin{itemize}\addtolength{\itemsep}{2pt plus 1pt minus 0.5pt}
    \item $S=\semigroup{3,7,11}$: $\mdim S =3$ by Theorem \ref{Theorem:m3}.
    
    \item $S=\semigroup{3,10,11}$: $\mdim S =2$ by Theorem \ref{Theorem:m3}.  
        
    \item $S=\semigroup{5,6,7,9}$: $\mdim S = 4$.  Since $5,6,7 \in S$ and $8 \notin S$, Lemma \ref{Lemma:Consecutive} ensures that $\mdim S \geq 4$.  The CMM ensures that $\mdim S = 4$ since $S =  \S\left( \left[ \begin{smallmatrix}\0 & \0 & -32 & \frac{1}{16} \\ 1 & \0 & -16 & \0 \\ \0 & \frac{1}{4} & -2 & \0 \\ \0 & \0 & -1024 & \0 \\ \end{smallmatrix}\right]\right)$.
        
    \item $S=\semigroup{5,6,9,13}$: $\mdim S = 3$.  Since $5,6\in S$ and $7 \notin S$, Lemma \ref{Lemma:Consecutive} ensures that $\mdim S \geq 3$.  The CMM ensures that $\mdim S = 3$ since $S =  \S\left(\left[
\begin{smallmatrix}
 \0 & -\frac{1}{8} & 64 \\
 64 & 4 & \0 \\
 \0 & \frac{1}{512} & \0 \\
\end{smallmatrix}
\right] \right)$.
        
    \item $S=\semigroup{5,7,9,11,13}$:
        $\mdim S =3$. Since $\s(S) = \{5,7\}$ and $\gcd(5,7)= 1$, Lemma \ref{Lemma:Small} ensures that $\mdim S \geq 3$.
        The CMM ensures that $\mdim S = 3$ since $S = \S\left( \left[ \begin{smallmatrix}
 \0 & \0 & \0 \\
 \frac{1}{64} & \0 & -6144 \\
 \0 & \frac{1}{128} & \0 \\
        \end{smallmatrix}\right]\right)$.
        
    \item $S=\langle5,9,11,12,13\rangle$:
        $\mdim S =2$.  Use Theorem \ref{Theorem:LatticeExamples}(c) with $a=5$ and $b=9$.
        
    \item $S=\langle 6,7,9,10,11\rangle$:
        $\mdim S =3$.  Since $6,7 \in S$ and $8 \notin S$, Lemma \ref{Lemma:Consecutive} ensures that $\mdim S \geq 3$. 
        The CMM ensures that $\mdim S = 3$ since $S = \S\left( \left[\begin{smallmatrix}
 \0 & \0 & \frac{1}{8} \\
 1024 & \0 & -128 \\
 \0 & \frac{1}{16} & 4 \\
                \end{smallmatrix}\right]\right)$.
    
    \item $S=\semigroup{6,9,10,11,13,14}$:
        $\mdim S =2$. Use Theorem \ref{Theorem:LatticeExamples}(c) with $a=6$ and $b=9$.

    \item $S=\semigroup{7,9,10,11,12,13,15}$:
        $\mdim S =2$. Use Theorem \ref{Theorem:LatticeExamples}(c) with $a=7$ and $b=9$.
        
    \item $S=\semigroup{9,10,11,12,13,14,15,16,17}$:
        $\mdim S =2$ by Theorem \ref{Theorem:LatticeExamples}(b) since $S = S_9$.
\end{itemize}

\bigskip\noindent\textbf{$F(S)=9$}
\begin{itemize}\addtolength{\itemsep}{2pt plus 1pt minus 0.5pt}
    \item $S=\semigroup{2,11}$: $\mdim S =2$.
    Use Theorem \ref{Theorem:LatticeExamples}(c) with $a=2$ and $b=11$.
    
    \item $S=\semigroup{4,6,7}$: $\mdim S =4$.
    Since $6,7,8 \in S$ and $9 \notin S$, Lemma \ref{Lemma:Consecutive} ensures that $\mdim S \geq 4$.  Theorem \ref{Theorem:NSRM2} with $d=4$ shows that $\mdim S= 4$.
        
    \item $S=\semigroup{4,6,11,13}$: $\mdim S =2$.  The CMM provides $S = \S\left(\left[\begin{smallmatrix}
        \0&64\\
        \frac{1}{32}&16
    \end{smallmatrix}\right]\right)$.
    
    \item $S=\semigroup{4,7,10,13}$:
    $\mdim S = 3$.
    Since $7,8 \in S$ and $9 \notin S$, Lemma \ref{Lemma:Consecutive} ensures that $\mdim S \geq 3$. The CMM ensures that $\mdim S = 3$ since $S = \S\left( \left[
        \begin{smallmatrix}
         \0 & -\frac{1}{4} & 256 \\
         16 & 2 & \0 \\
         \0 & \frac{1}{512} & \0 \\
        \end{smallmatrix}
        \right] \right)$.

    \item $S=\semigroup{4,10,11,13}$: $\mdim S =2$. Use Theorem \ref{Theorem:LatticeExamples}(c) with $a=4$ and $b=10$.
    
   \item $S=\semigroup{5,6,7}$: $\mdim S = 4$.
        Since $5,6,7\in S$ but $8\notin S$, Lemma \ref{Lemma:Consecutive} ensures that $\mdim S \geq 4$.  The CMM ensures that $\mdim S=4$ since $S = \S\left( \left[
            \begin{smallmatrix}
             \0 & \0 & -64 & \frac{1}{16} \\
             2 & \0 & -64 & 0 \\
             \0 & \frac{1}{16} & -2 & 0 \\
             \0 & \0 & -2048 & 0 \\
            \end{smallmatrix}
            \right]\right)$.

    \item $S=\semigroup{5,6,7,8}$: $\mdim S =5$.  Since $5,6,7,8 \in S$ and $9 \notin S$, Lemma \ref{Lemma:Consecutive} ensures that $\mdim S \geq 5$.  Theorem \ref{Theorem:NSRM2} with $d=5$ shows that $\mdim S= 5$.
    
    \item $S=\langle5,6,8\rangle$: $\mdim S=3$. Since $5,6\in S$ but $7\notin S$, Lemma \ref{Lemma:Consecutive} ensures $\mdim S\geq 3$. Since $S=\S\bigg(\left[ \begin{smallmatrix}
 -12 & -196608 & -768 \\
 \frac{2363}{4096} & 9448 & \frac{1181}{32} \\
 -\frac{273}{4} & -1379072 & -5388 \\
\end{smallmatrix}\right]\bigg)$, we have $\mdim S=3.$

    \item $S=\langle5,6,13,14\rangle$: $\mdim S=3$. Since $5,6\in S$ but $7\notin S$, Lemma \ref{Lemma:Consecutive} ensures $\mdim S\geq 3$. The CMM ensures that $\mdim S=3$ since $S=\S\bigg(\left[ \begin{smallmatrix}
 \0 & \0 & 256 \\
 \frac{1}{4096} & \0 & -\frac{1}{32} \\
 \0 & 256 & 4 \\
\end{smallmatrix}\right]\bigg)$.

       \item $S=\langle5,7,8,11\rangle$: $\mdim S=3$. Since $7,8\in S$ but $9\notin S$, Lemma \ref{Lemma:Consecutive} ensures $\mdim S\geq 3$. Since $S=\S\bigg(\left[
\begin{smallmatrix}
 -\frac{317}{4} & \frac{1}{32} & -628 \\
 -1286 & -\frac{1}{4} & -10336 \\
 \frac{837}{32} & \frac{3}{256} & \frac{419}{2} \\
\end{smallmatrix}
\right]\bigg)$, we have $\mdim S=3$.

    \item $S=\semigroup{5,7,11,13}$: $\dim S = 3$. Since $\s(S)=\{5,7\}$ and $\gcd(5,7)=1$, Lemma \ref{Lemma:Small} ensures that $\mdim S\geq 3$. The CMM ensures that $\mdim S=3$ since 
    $S = \S\left( \left[
\begin{smallmatrix}
 \0 & \0 & \frac{1}{16} \\
 \frac{1}{16} & \0 & -\frac{5}{512} \\
 \0 & 2048 & -12 \\
\end{smallmatrix}
\right] \right)$.

    \item $S=\langle5,8,11,12,14\rangle$: $\mdim S=3.$ Since $\s(S)=\{5,8\}$ and $\gcd(5,8)=1$, Lemma \ref{Lemma:Small} ensures that $\mdim S\geq 3$. The CMM ensures that $\mdim S=3$ since
            $S=\S\bigg(\left[
            \begin{smallmatrix}
             \0 & -\frac{1}{16} & \frac{1}{512} \\
             64 & 2 & \0 \\
             \0 & \0 & \0 \\
            \end{smallmatrix}
            \right]\bigg)$.
         
    \item $S=\langle5,11,12,13,14\rangle$: $\mdim S =2$ 
    Use Theorem \ref{Theorem:LatticeExamples}(c) with $a=5$ and $b=11$.
    \item$S=\langle6,7,8,10,11\rangle$: $\mdim S =4$. Since $6,7,8\in S$ but $9\notin S$, Lemma \ref{Lemma:Consecutive} ensures $\mdim S\geq4$. The CMM ensures that $\mdim S=4$ since $S=\S\bigg(\left[
\begin{smallmatrix}
 \0 & \0 & \0 & 2048 \\
 \frac{1}{4} & \0 & \0 & 256 \\
 \0 & 2 & \0 & \0 \\
 \0 & \0 & \frac{1}{64} & -2 \\
\end{smallmatrix}
\right]\bigg)$.
    
    \item $S=\langle6,7,10,11,15\rangle$: $\mdim S =3$. Since $6,7\in S$ but $8\notin S$, Lemma \ref{Lemma:Consecutive} ensures $\mdim S\geq 3.$ The CMM ensures that $\mdim S=3$ since
        $S =\S\bigg(\left[
        \begin{smallmatrix}
         \0 & \0 & 8388608 \\
         \frac{1}{256} & \0 & -16384 \\
         \0 & \frac{1}{512} & 8 \\
        \end{smallmatrix}
        \right]\bigg)$.

   \item $S=\langle6,8,10,11,13,15\rangle$: $\mdim S =2$. The CMM ensures that $\mdim S=2$ since $S =\S\bigg(\left[
        \begin{smallmatrix}
         \0 & -69632  \\
         \frac{1}{1024} & -64 \\
        \end{smallmatrix}
        \right]\bigg)$.
    
    \item $S=\langle6,10,11,13,14,15\rangle$:
    $\mdim S =2$. Use Theorem \ref{Theorem:LatticeExamples}(c) with $a=6$ and $b=10$.

    \item $S=\langle7,8,10,11,12,13\rangle$: $\mdim S =3$. Since $7,8\in S$ but $9\notin S$,  Lemma \ref{Lemma:Consecutive} ensures that $\mdim S \geq 3$. The CMM ensures that $\mdim S=3$ since 
    $S =\S\bigg(\left[
        \begin{smallmatrix}
         \0 & \0 & 2048 \\
         \frac{1}{8} & \0 & -128 \\
         \0 & \frac{1}{32} & -12 \\
        \end{smallmatrix}
        \right]\bigg)$. 
    
    \item $S=\langle7,10,11,12,13,15,16\rangle$: $\mdim S =2$. Use Theorem \ref{Theorem:LatticeExamples}(c) with $a=7$ and $b=10$.
    
    \item $S=\langle8,10,11,12,13,14,15,17\rangle$: $\mdim S =2$. Use Theorem \ref{Theorem:LatticeExamples}(c) with $a=8$ and $b=10$.
    
    \item$S=\langle10,11,12,13,14,15,16,17,18,19\rangle$: $\mdim S =2$ by Theorem \ref{Theorem:LatticeExamples}(b) since $S = S_{10}$.
\end{itemize}

\bigskip\noindent\textbf{$F(S)=10$}
\begin{itemize}\addtolength{\itemsep}{2pt plus 1pt minus 0.5pt}
    \item $S=\langle3,8,13\rangle$: 
        $\mdim S =3$. Since $8,9 \in S$ and $10 \notin S$, Lemma \ref{Lemma:Consecutive} ensures that $\mdim S \geq 3$.  Theorem \ref{Theorem:NSRM2} with $d=3$ shows that $\mdim S= 3$.
        
    \item $S=\langle3,11,13\rangle$: 
        $\mdim S =2$. Use Theorem\ref{Theorem:LatticeExamples}(c) with $a=3$ and $b=11$.
    
    \item $S=\langle4,7,9\rangle$: 
        $\mdim S =4$. Since $7,8,9 \in S$ and $10 \notin S$, Lemma \ref{Lemma:Consecutive} ensures that $\mdim S \geq 4$.  Theorem \ref{Theorem:NSRM2} with $d=4$ shows that $\mdim S= 4$.
        
    \item $S=\langle4,7,13\rangle$:
        $\mdim S =3$. Since $7,8\in S$ but $9\notin S$, Lemma \ref{Lemma:Consecutive} ensures $\mdim S\geq 3$. The CMM ensures that $\mdim S=3$ since $S=\S\bigg(\left[
\begin{smallmatrix}
 \0 & -\frac{1}{64} & 32 \\
 256 & 2 & \0 \\
 \0 & \frac{1}{1024} & \0 \\
\end{smallmatrix}
\right]\bigg)$.
   
   \item $S=\langle4,9,11,14\rangle$:
        $\mdim S =3$. Since $8,9\in S$ but $10\notin S$, Lemma \ref{Lemma:Consecutive} ensures $\mdim S\geq 3$. The CMM ensures that $\mdim S=3$ since $S=\S\bigg(\left[
\begin{smallmatrix}
 \0 & -16 & \frac{1}{128} \\
 \frac{1}{4} & 2 & \0 \\
 \0 & 4096 & \0 \\
\end{smallmatrix}
\right]\bigg)$.
    \item $S=\langle4,11,13, 14\rangle$: 
        $\mdim S =2$. Use Theorem \ref{Theorem:LatticeExamples}(c) with $a=4$ and $b=11$.
    
    \item $S=\langle6,7,8,9,11\rangle$:
        $\mdim S =5$. Since $6,7,8,9\in S$ but $10\notin S$, Lemma \ref{Lemma:Consecutive} ensures $\mdim S \geq5$. The CMM ensures that $\mdim S=5$ since $S=\S\bigg( \left[
\begin{smallmatrix}
 \0 & \0 & \0 & -32 & 4096 \\
 \frac{1}{2} & \0 & \0 & 8 & 0 \\
 \0 & \frac{1}{16} & \0 & -\frac{1}{4} & 0 \\
 \0 & \0 & 16 & 2 & 0 \\
 \0 & \0 & \0 & \frac{1}{64} & 0 \\
\end{smallmatrix}
\right]\bigg)$.

    \item $S=\langle6,7,8,11\rangle$:
        $\mdim S =4$. Since $6,7,8\in S$ but $9\notin S$, Lemma \ref{Lemma:Consecutive} ensures $\mdim S\geq 4$.  
    The CMM ensures that $\mdim S=4$ since $S=\S\bigg(\left[
\begin{smallmatrix}
 \0 & \0 & 64 & \frac{1}{32} \\
 1 & \0 & \0 & \0 \\
 \0 & \frac{1}{8} & -2 & \0 \\
 \0 & \0 & 4096 & \0 \\
\end{smallmatrix}
\right]\bigg)$.

    \item $S=\langle6,7,9,11\rangle$:
        $\mdim S=3$. Since $6,7\in S$ but $8\notin S$, Lemma \ref{Lemma:Consecutive} ensures $\mdim S\geq 3$. Then $\mdim S=3$ since $S=\S\bigg(\left[
\begin{smallmatrix}
 \frac{337}{16}&  \frac{337}{4}& 1348 \\
 -\frac{529}{32} & -\frac{561}{8} & -1122 \\
 -\frac{3247}{1024} & -\frac{6767}{256} & -\frac{7279}{16} \\
\end{smallmatrix}
\right]\bigg)$.
        
    \item $S=\langle6,7,11,15,16\rangle$: $\mdim S=3.$
       Since $6,7\in S$ but $8\notin S$, Lemma \ref{Lemma:Consecutive} ensures $\mdim S\geq 3$.  
The CMM ensures that $\mdim S=3$ since $S=\S\bigg(\left[
\begin{smallmatrix}
 \0 & \0 & 16777216 \\
 \frac{1}{64} & \0 & -131072 \\
 \0 & \frac{1}{4096} & 8 \\
\end{smallmatrix}
\right]\bigg)$.

    \item $S=\langle6,8,9,11,13\rangle$:
        $\mdim S =3$. Since $8,9\in S$ but $10\notin S$, Lemma \ref{Lemma:Consecutive} ensures $\mdim S \geq 3$.
        The CMM ensures that $\mdim S=3$ since
        $S =\S\bigg(\left[
        \begin{smallmatrix}
         \0 & \0 & -56 \\
         \frac{1}{16} & \0 & 1 \\
         \0 & 4 & 16 \\
        \end{smallmatrix}
        \right]\bigg)$.
       
    \item $S=\langle6,8,11,13,15\rangle$:
        Observe that $\s(S)=\{6,8\}$ and $\gcd \s(S) = 2$.  Since $\mdim S/2 = \mdim \semigroup{3,4} = 3$ by \cite[~Ex. 5.2]{NSRM1}.
        Theorem \ref{Theorem:Quotient} ensures that $\mdim S \geq 3$.
        The CMM ensures that $\mdim S=3$ since $S=\S\bigg(\left[
\begin{smallmatrix}
 \0 & \0 & 128 \\
 \frac{1}{1024} & \0 & -\frac{1}{8} \\
 \0 & 64 & 4 \\    
\end{smallmatrix}
\right]\bigg)$.

    \item $S=\langle6,9,11,13,14,16\rangle$:
        $\mdim S=2$.
        The CMM provides $S=\S\bigg(\left[
\begin{matrix}
    \0&\frac{7}{1024}\\
    4096&-6\\
\end{matrix}
\right]\bigg)$.

    \item $S=\langle6,11,13,14,15,16\rangle$:
        $\mdim S =2$. Use Theorem \ref{Theorem:LatticeExamples}(c) with $a=6$ and $b=11$.
    
    \item $S=\langle7,8,9,11,12,13\rangle$:
        $\mdim S=4$. Since $7,8,9\in S$ but $10\notin S$, Lemma \ref{Lemma:Consecutive} ensures $\mdim S\geq4$. The CMM ensures that $\mdim S=4$ since $S=\S\bigg(\left[
\begin{smallmatrix}
 \0 & \0 & \0 & \frac{1}{4} \\
 4 & \0 & \0 & \0 \\
 \0 & 4 & 4 & \0 \\
 \0 & \0 & \frac{1}{2} & 2 \\
\end{smallmatrix}
\right]\bigg)$.
        
    \item $S=\langle7,8,11,12,13,17\rangle$:
        $\mdim S=3$. Since $7,8\in S$ but $8\notin S$, Lemma \ref{Lemma:Consecutive} ensures $\mdim S\geq 3$. The CMM ensures that $\mdim S=3$ since $S = \S\left( \left[
    \begin{smallmatrix}
    \0 & \0 & -4096 \\
    \frac{1}{32} & \0 & -128 \\
    \0 & \frac{1}{16} & -10 \\
    \end{smallmatrix}
    \right]\right)$.
   
    \item $S=\langle7,9,11,12,13,15,17\rangle$:
        $\mdim S =3$. Since $\s(S)=\{7,9\}$ and $\gcd(7,9)=1$, Lemma \ref{Lemma:Small} ensures $\mdim S\geq3$. The CMM ensures that $\mdim S=3$ since $S = \S\left( \left[
    \begin{smallmatrix}
    \0 & \0 & -\frac{3}{4} \\
    256 & \0 & -208 \\
    \0 & \frac{1}{8} & 32 \\
    \end{smallmatrix}
    \right]\right)$.

    \item $S=\langle7,11,12,13,15,16,17\rangle$:
        $\mdim S =2$. Use Theorem \ref{Theorem:LatticeExamples}(c) with $a=7$ and $b=11$.
        
    \item $S=\langle8,9,11,12,13,14,15\rangle$:
        $\mdim S =3$. Since $8,9\in S$ but $10\notin S$, Lemma \ref{Lemma:Consecutive} ensures $\mdim S\geq 3$. The CMM ensures that $\mdim S=3$ since $S = \S\left( \left[
    \begin{smallmatrix}
    \0 & \0 & 20 \\
    \frac{1}{16} & \0 & \frac{3}{2} \\
    \0 & 8 & 18 \\
    \end{smallmatrix}
    \right]\right)$.
        
    \item $S=\langle8,11,12,13,14,15,17,18\rangle$:
     $\mdim S =2$. Use Theorem \ref{Theorem:LatticeExamples}(c) with $a=8$ and $b=11$.
     
    \item $S=\langle9,11,12,13,14,15,16,17,19\rangle$: $\mdim S =2$. Use Theorem \ref{Theorem:LatticeExamples}(c) with $a=9$ and $b=11$.
    
    \item $S=\langle11,12,13,14,15,16,17,18,19,20,21\rangle$: $\mdim S =2$ by Theorem \ref{Theorem:LatticeExamples}(b) since $S = S_{11}$.

\end{itemize}
}
\bigskip\noindent\textbf{Matricial dimension by genus.}
Most of the semigroups with genus at most $6$ are already included in the list above.  The exceptions are the following.

\begin{itemize}\addtolength{\itemsep}{2pt plus 1pt minus 0.5pt}
    \item $S=\semigroup{2,13}$: $\mdim S=2$. Use Theorem \ref{Theorem:LatticeExamples}(c) with $a=2$ and $b=13$.
    \item $S=\semigroup{3,7}$: $\mdim S=3$ by Theorem \ref{Theorem:m3}.
    \item $S=\semigroup{4,5}$: $\mdim S=4$. Since $8,9,10\in S$ but $11\notin S$, Lemma \ref{Lemma:Consecutive} ensures $\mdim S\geq4$. Theorem \ref{Theorem:NSRM2} with $d=4$ shows that $\mdim S= 4$.
    \item $S=\semigroup{4,6,9}$: $\mdim S=4$. Since $8,9,10\in S$ but $11\notin S$, Lemma \ref{Lemma:Consecutive} ensures $\mdim S\geq4$. Theorem \ref{Theorem:NSRM2} with $d=4$ shows that $\mdim S= 4$.
    \item $S=\semigroup{5,7,8,9}$: $\mdim S=5$. Since $7,8,9,10\in S$ but $11\notin S$, Lemma \ref{Lemma:Consecutive} ensures $\mdim S\geq5$. Theorem \ref{Theorem:NSRM2} with $d=5$ shows that $\mdim S= 5$.
    \item $S=\semigroup{6,7,8,9,10}$: $\mdim S=6$. Since $6,7,8,9,10\in S$ but $11\notin S$, Lemma \ref{Lemma:Consecutive} ensures $\mdim S\geq 6$. Theorem \ref{Theorem:NSRM2} with $d=6$ shows that $\mdim S= 6$.
    \end{itemize}

\section{Open questions}\label{Section:Open}

A natural extension of this work is to extend the tabulation of matricial dimensions beyond $F(S) \leq 10$ and $g(S) \leq 6$. How far can this be extended before new techniques are required?  A first step in this direction is the following question.

\begin{prob}
    Compute the matricial dimensions of all semigroups with $F(S) = 11$.
\end{prob}

In a similar manner, one wishes to compute the matricial dimensions of numerical semigroups with $m(S) = 4$ and beyond. Beyond $m(S) = 3$, it becomes prudent to work with Ap\'ery sets instead of generators, as we did in the proof of Theorem \ref{Theorem:m3}.

\begin{prob}
    Compute the matricial dimensions of all semigroups with $m(S) = 4$.
\end{prob}

The semigroup $S=\semigroup{5,7,16,18}$, which has Frobenius number $13$, appeared in \cite{NSRM3} as a counterexample to the lonely element conjecture.  It is known that $3 \leq \mdim S \leq 5$.

\begin{prob}
    What is the matricial dimension of $\semigroup{5,7,16,18}$?   
\end{prob}

\bibliography{MainDocument}

\providecommand{\bysame}{\leavevmode\hbox to3em{\hrulefill}\thinspace}
\providecommand{\MR}{\relax\ifhmode\unskip\space\fi MR }
\providecommand{\MRhref}[2]{%
  \href{http://www.ams.org/mathscinet-getitem?mr=#1}{#2}
}
\providecommand{\href}[2]{#2}
\begin{thebibliography}{1}

\bibitem{Assi}
Abdallah Assi, Marco D'Anna, and Pedro~A. Garc\'{\i}a-S\'{a}nchez,
  \emph{Numerical semigroups and applications}, RSME Springer Series, vol.~3,
  Springer, Cham, [2020] \copyright 2020, Second edition [of 3558713].
  \MR{4230109}

\bibitem{MR4291927}
Scott Chapman, Rebecca Garcia, and Christopher O'Neill, \emph{Beyond coins,
  stamps, and {C}hicken {M}c{N}uggets: an invitation to numerical semigroups},
  A project-based guide to undergraduate research in mathematics---starting and
  sustaining accessible undergraduate research, Found. Undergrad. Res. Math.,
  Birkh\"{a}user/Springer, Cham, [2020] \copyright 2020, pp.~177--202.
  \MR{4291927}

\bibitem{MR4261927}
Scott Chapman, Pedro Garc\'{\i}a-S\'{a}nchez, and Christopher O'Neill,
  \emph{Distances between factorizations in the {C}hicken {M}c{N}ugget monoid},
  College Math. J. \textbf{52} (2021), no.~3, 158--176. \MR{4261927}

\bibitem{MR3885968}
Scott~T. Chapman and Chris O'Neill, \emph{Factoring in the {C}hicken
  {M}c{N}ugget monoid}, Math. Mag. \textbf{91} (2018), no.~5, 323--336.
  \MR{3885968}

\bibitem{NSRM3}
Arsh Chhabra and Stephan~Ramon Garcia, \emph{Numerical semigroups from rational
  matrices {III}: semigroups of matricial dimension two and a counterexample to
  the lonely element conjecture}, Comm. Alg., in press.
  \url{https://doi.org/10.1080/00927872.2025.2537271}.

\bibitem{NSRM2}
Arsh Chhabra, Stephan~Ramon Garcia, and Christopher O'Neill, \emph{Numerical
  semigroups from rational matrices {II}: matricial dimension does not exceed
  multiplicity}, Bull. Aust. Math. Soc. \textbf{112} (2025), no.~1, 155--162.
  \MR{4931780}

\bibitem{NSRM1}
Arsh Chhabra, Stephan~Ramon Garcia, Fangqian Zhang, and Hechun Zhang,
  \emph{Numerical semigroups from rational matrices {I}: power-integral
  matrices and nilpotent representations}, Comm. Algebra \textbf{53} (2025),
  no.~3, 1127--1137. \MR{4865362}

\bibitem{Rosales}
J.~C. Rosales and P.~A. Garc\'{\i}a-S\'{a}nchez, \emph{Numerical semigroups},
  Developments in Mathematics, vol.~20, Springer, New York, 2009. \MR{2549780}

\end{thebibliography}
\bibliographystyle{amsplain}

\end{document}